\documentclass[11pt,a4paper]{amsart}
\usepackage{a4}
\usepackage{amsfonts, amssymb}
\usepackage[all]{xy}

\newtheorem{thm}{Theorem}

\newtheorem{defn}{Definition}

\numberwithin{equation}{section} \numberwithin{thm}{section}
\numberwithin{lem}{section} \numberwithin{problem}{section}
\numberwithin{cor}{section}

\begin{document}
\title{A greedy algorithm for $B_h[g]$ sequences}
\begin{abstract} For any positive integers $h\ge 2$ and $g\ge 1$, we present a greedy algorithm that provides an infinite $B_h[g]$ sequence with $a_n\le 2gn^{h+(h-1)/g}.$

\end{abstract}
\author{Javier Cilleruelo}
\address{Instituto de Ciencias Matem\'aticas (ICMAT) and Departamento de Matem\'aticas\\
Unversidad Aut\'onoma de Madrid\\
Madrid 28049\\
Espa\~na.
}
\email{franciscojavier.cilleruelo@uam.es}


\maketitle

\section{Introduction}
Given positive integers $h\ge 2$ and $g\ge 1$, we say that a sequence of integers $A$ is a $B_h[g]$ sequence if the number of representations of any integer $n$ in the form
$$n=a_1+\dots+a_h,\quad a_1\le \cdots \le a_h,\quad a_i\in A$$
is bounded by $g$.  The $B_h[1]$ sequences are  simply called $B_h$ sequences.

 A trivial counting argument shows that if $A=\{ a_n\}$ is a $B_h[g]$ sequence then $a_n\gg n^h$. On the other hand, the greedy algorithm introduced by Erd\H os \footnote{This algorithm has been atributed to Mian and Chowla, but it seems (see \cite{E1}) that was Erd\H os who first used this algorithm.} provides an infinite $B_h$ sequence with $a_n\le 2n^{2h-1}.$ 
 
 {\bf Classic greedy algorithm}: {\it Let $a_1=1$ and for $n\ge 2$, define $a_n$ as the smallest positive integer, greater than $a_{n-1}$, such that $a_1,\dots, a_n$ is a $B_h[g]$ sequence.}
 
 When $g=1$, the greedy algorithm defines $a_1=1,\ a_2=2$ and for $n\ge 3$, defines $a_{n}$ as the smallest positive integer that is not of the form
$$\frac 1k\big (a_{i_1}+\cdots +a_{i_h}-(a_{i_1'}\cdots+a_{i_{h-k}'})\big )$$ for any $1\le i_1,\dots,i_h,i_1',\dots,i_{h-k}'$ $\le n-1$ and $1\le k\le h-1$. Since there are at most $(n-1)^{2h-1}+\cdots +(n-1)^{h+1}\le (n-1)^{2h}/(n-2)$ forbidden elements for $a_{n}$, then $a_{n}\le 1+(n-1)^{2h}/(n-2)\le 2n^{2h-1}.$

It is possible that the classic greedy algorithm  may provide a denser sequence when $g>1$, but it is not clear how to prove it. For this reason other methods have been used  to obtain dense infinite $B_h[g]$ sequences:

\noindent {\bf Theorem A.} {\it
	Given $h\ge 2$ and $g\ge 1$, there exists an infinite $B_h[g]$ sequence with $a_n\ll n^{h+\delta}$ with $\delta=\delta_h(g)\to 0$ when $g\to \infty.$}

Erd\H os and Renyi \cite{ER} proved Theorem A for $h=2$ using the probabilistic method. Ruzsa gave the first proof for any $h\ge 3$ (a sketch of that proof, which consists in an explicit construction, appeared in \cite{EF} and a detailed proof in \cite{CKRV}).

The aim of this paper is to describe a distinct greedy algorithm that provides a $B_h[g]$ sequence that grows slower than all previous known constructions for $g>1$. More specifically, Theorem \ref{main} gives an easy proof of Theorem A with $\delta_h(g)=(h-1)/g.$ 

In the table below we resume all  previous results on this problem for $g>1$ expressed in form  $a_n\ll n^{h+\delta_h(g)}$ and the method used in each case.  The probabilistic method, which we denote by PM, has been used in  most of the constructions.
\smallskip
$$\begin{array}{|cc|c|}\hline \delta_2(g)&\le 2/g+o_n(1) \qquad\quad & \text{PM \cite{ER}}  \\ \hline \delta_2(g)&\le 1/g+o_n(1)\qquad \quad& \text{PM + alteration method \cite{Ci3}}\\ \hline \delta_3(g)&\le 2/g+\epsilon,\ \ \ \epsilon>0 \quad  & \text{PM+ combinatorial ingredients \cite{CKRV}} \\ \hline\delta_h(g)&\ll_h 1/(\log g \log \log g)&  \text{Explicit construction, Ruzsa \cite{EF},\cite{CKRV}} \\ \hline\delta_h(g)&\ll_h 1/g^{1/(h-1)}\qquad \quad &\text{ PM+ Kim-Vu method \cite{KV}} \\ \hline \delta_h(g)& \ll 2^hh (h!)^2/g \qquad \quad & \text{PM + Sunflower Lemma \cite{CKRV}}\\ \hline \delta_h(g)&\le (h-1)/g \qquad \qquad &\text{New greedy algorithm, Theorem \ref{main}} \\ \hline\end{array}$$

For $g=1$ there are special constructions of $B_h$ sequences with slower growth.
\smallskip
$$\begin{array}{|cc|c|}\hline \delta_h(1)&\le h-1 \qquad\qquad\qquad\qquad\qquad\qquad\quad& \text{Classic greedy algorithm} \\ \hline \delta_2(1)&\le 1-\epsilon_n,\ \ \epsilon_n=\log \log n/ \log n \qquad\quad & \quad \text{PM + graph tools \cite{AKS} }\\ \hline \delta_2(1)&\le \sqrt 2-1+o_n(1) \qquad \qquad\qquad\qquad\quad& \text{Real log method + PM \cite{Ru2}} \\ \hline \delta_2(1)&\le \sqrt 2-1+o_n(1) \qquad \qquad\qquad\qquad\quad& \text{Explicit construction \cite{Ci4}}\\  \hline \delta_h(1)&\le \sqrt{(h-1)^2+1}-1+o_n(1),\ h=3,4& \text{Gaussian arg method  + PM \cite{Ci5}}  \\ \hline \delta_h(1)&\le \sqrt{(h-1)^2+1
}-1+o_n(1),\ h\ge 3\quad  & \text{Discrete log method + PM \cite{Ci4}}\\ \hline\end{array}$$
\section{A new  greedy algorithm}
We need to introduce the notion of {\it strong} $B_h[g]$ set.

\begin{defn}
	We say that $A_n=\{a_1,\dots,a_n\}$ is a strong $B_h[g]$ set if the following conditions are satisfied:
	\begin{itemize}
		\item[i)] $A_n$ is a $B_h[g]$ set.
		\item[ii)] $|\{x: r_{A_n}(x)\ge s\}|\le n^{h+(1-s)(h-1)/g},$ for $s=1,\dots,g$, where
	\end{itemize}
 $$r_{A_n}(x)=|\{(a_{i_1},\dots a_{i_h}):\quad 1\le i_1\le \cdots \le i_h\le n,\quad x=a_{i_1}+\cdots +a_{i_h}\}|.$$
\end{defn}

\begin{thm}\label{main}Let $a_1=1$ and for $n\ge 1$
	define $a_{n+1}$ as the smallest positive integer, distinct to $a_1,\dots,a_n$, such that
	$a_1,\dots,a_{n+1}$ is a strong $B_h[g]$ set. 	The infinite sequence $A=\{a_n\}$ given by this greedy algorithm  is a $B_h[g]$ sequence with $a_n\le 2gn^{h+(h-1)/g}.$
\end{thm}
\begin{proof} Let $a_1=1,\ a_2=2$ and suppose that $A_n=\{a_1,\dots,a_n\}$  is the strong $B_h[g]$ set given by this algorithm for some $n\ge 2$. 	We will find an upper bound for the number of forbiden positive integers for $a_{n+1}$. We  use the notation $R_s(A_n)=|\{x: r_{A_n}(x)\ge s\}|$
to classify the forbidden elements $m$ in the following sets:
	\begin{itemize}\item[i)]	$F_n=\{m:\ m\in A_n \}$.
		\item[ii)]$F_{0,n}=\{m:\ A_n\cup m \text{ is not a } B_h[g] \text{ set}\}$
		\item[iii)]$F_{s,n}=\{m:\ R_{s}(A_n\cup m)>(n+1)^{h+(1-s)(h-1)/g}\},\quad s=1,\dots,g$.
			\end{itemize}

Hence $a_{n+1}$ is the smallest positive integer not belonging to  $\left (\bigcup_{s=0}^gF_{s,n}\right )\cup F_n$ and then the proof of Theorem \ref{main} will be completed if we prove that  \begin{equation}\label{Upper}\left |\left (\bigcup_{s=0}^gF_{s,n}\right )\cup F_n\right |\le 2g(n+1)^{h+(h-1)/g}-1.\end{equation}
	
	It is clear that $|F_n|=n$. Next, we find an upper bound for the cardinality of $F_{s,n},\ s=0,\dots,g$.
	
	 The elements of $F_{0,n}$ are the positive integers of the form $\frac 1k\left (x-(a_{i_1}+\cdots +a_{i_{h-k}})\right )$ for some $1\le i_1,\dots,i_{h-k}\le n,\ 1\le k\le h-1$ and for some $x$ with $r_{A_n}(x)= g$. Thus,
	\begin{eqnarray*}\label{F0}|F_{0,n}|&\le &(n^{h-1}+\cdots +n+1)|\{x:\ r_{A_n}(x)= g\}|\\ &\le & n^{h}/(n-1)\ R_{g}(A_n)\\ &\le & 2n^{h-1}n^{1+(h-1)/g}=2n^{h+(h-1)/g}.\end{eqnarray*}
	
	For $s=1$, note that $R_1(A_n\cup m)\le (n+1)^h$ for any $m$, so $|F_{1,n}|=0$. 
	
	For $s=2,\dots,g$, and for any $m$ we have
	\begin{equation}\label{R}R_{s}(A_n\cup m)\le R_{s}(A_n)+T_{s,n}(m),\end{equation}
	where
	$$T_{s,n}(m)=|\big \{x:\ r_{A_n}(x)\ge s-1,\ x\in km+ A_n+\stackrel{h-k}\cdots+ A_n \text{ for some } 1\le k\le h\big \}|.$$ In the  case $k=h$, the expression $x\in km+A_n+\stackrel{h-k}\cdots+ A_n$  means $x=hm$.
	
	We observe that if $T_{s,n}(m)\le n^{h-1+(1-s)(h-1)/g}$, using \eqref{R} and that $A_n$ is a strong $B_h[g]$ set, we have
	\begin{eqnarray*}R_{s}(A_n\cup m)&\le & n^{h+(1-s)(h-1)/g}+n^{h-1+(1-s)(h-1)/g}\\ &\le& (n+1)^{h+(1-s)(h-1)/g}\end{eqnarray*} and then $m\not \in F_{s,n}$. 	Thus,
	\begin{eqnarray}\label{TT}\sum_mT_{s,n}(m)&\ge & \sum_{m\in F_{s,n}}T_{s,n}(m)  > n^{h-1+(1-s)(h-1)/g}|F_{s,n}| .\end{eqnarray}	
	
On the other hand, 	when we sum $T_{s,n}(m)$ over all $m$, each $x$ with $r_{A_n}(x)\ge s-1$ is counted no more than $ |A_n+\stackrel{h-1}\cdots+ A_n|+\cdots +|A_n|+1\le n^{h-1}+\cdots +n+1$ times. Then
	\begin{eqnarray}\label{TTT}\sum_mT_{s,n}(m)&\le & (1+n+\cdots +n^{h-1})R_{s-1}(A_n)\\ &\le & \frac{n^h-1}{n-1}n^{h+(2-s)(h-1)/g}.\nonumber\end{eqnarray}
	
Inequalities \eqref{TT} and \eqref{TTT} imply
	\begin{equation}\label{Fs}|F_{s,n}|\le \frac{n^h-1}{n-1}n^{1+(h-1)/g}\le 2n^{h+(h-1)/g}.\end{equation}
	
	Taking into account \eqref{F0}, the inequalities \eqref{Fs} for $s=2,\dots,g$ and the estimate $|F_n|=n$, we get \begin{eqnarray*}\left |\left (\bigcup_{s=0}^gF_{s,n}\right )\cup F_n \right |&\le &2n^{h+(h-1)/g}+ 2(g-1)n^{h+(h-1)/g}+n\\ &= & 2gn^{h+(h-1)/g}+n\le  2g(n+1)^{h+(h-1)/g}-1, \end{eqnarray*} which,  according to \eqref{Upper}, finishes the proof.
\end{proof}


\begin{thebibliography}{9}
\bibitem{AKS} M. Ajtai, J. Koml\'{o}s, and E. Szemer\'{e}di, \textit{A dense infinite Sidon sequence}, European J.
Combin. 2 (1981), 1--11.
\bibitem{Ci3} J. Cilleruelo, \emph{Probabilistic constructions of $B_2[g]$ sequences},    Acta Mathematica Sinica    26 (2010), no. 7, 1309--1314.
\bibitem{Ci4} J. Cilleruelo, \emph{Infinite Sidon sequences},  Advances in Mathematics 255 (2014), 474--486.
\bibitem{Ci5} J. Cilleruelo and R. Tesoro, \emph{Dense infinite $B_h$ sequences}, Publicacions Matematiques,  vol 59, nº1 (2015).
\bibitem{CKRV} J. Cilleruelo, S. Kiss, I. Ruzsa and C. Vinuesa, \emph{Generalization of a theorem of Erdos and Renyi on Sidon sets},   Random Structures and Algorithms,  vol 37, nº4 (2010)

\bibitem{E1} P. Erd{\H{o}}s, \emph{Solved and unsolved problems in combinatorics
	and combinatorial number theory} Congressus Numeratium, Vol . 32 (1981), pp . 49--62.
\bibitem{EF} P. Erd{\H{o}}s and Freud, \emph{On Sidon sequences and related problems},
Mat. Lapok 1 (1991), 1--44.
\bibitem{ER}
P. Erd\H os and A.Renyi, \emph{ Additive properties of random sequences of positive
integers}. Acta Arithmetica. 6 (1960) 83--110.
\bibitem{KV} J.H. Kim and  V. Vu,\emph{Concentration of multivariante Polynomials and its applications} Combinatorica 20 (3) (2000) 417--434.
\bibitem{Ru2} I. Ruzsa, \emph{ An infinite Sidon sequence.} J. Number Theory 68 (1998), no. 1, 63--71.

\end{thebibliography}
\end{document}